\documentclass[6pt]{article}
\usepackage[english]{babel}

\usepackage{amsmath, amsthm}
\usepackage{amscd}
\usepackage{amsfonts}
\usepackage{amssymb}
\usepackage{mathrsfs}
\usepackage{url}
\usepackage[toc,page]{appendix}
\numberwithin{equation}{section}

\title{Cousin groups and Hodge structures}
\author{Rodion N. D\'eev}

\newcommand{\fg}{\mathfrak{g}}
\DeclareMathOperator{\End}{End}
\newcommand{\Z}{\mathbb{Z}}
\newcommand{\bC}{\mathbb{C}}
\DeclareMathOperator{\id}{id}
\newcommand{\R}{\mathbb{R}}
\DeclareMathOperator{\GL}{GL}
\DeclareMathOperator{\Gal}{Gal}
\newcommand{\Q}{\mathbb{Q}}
\DeclareMathOperator{\Hom}{Hom}

\newtheorem{pr}{Proposition}[section]

\newtheorem*{probl}{Problem}
\theoremstyle{definition}
\newtheorem{defn}{Definition}

\begin{document}
\maketitle
\begin{abstract}
We consider the geometric properties of Hodge Cousin groups, introduced in an unpublished paper \cite{OVV}, emphasizing the case of Hodge Cousin groups corresponding to polarized $\Q$-Hodge structures. Basing on this consideration, we introduce the class of abelian Cousin groups and prove an analogue of Poincar\'e complete reducibility theorem for them.
\end{abstract}
\tableofcontents

\newpage

\section{Introduction}

From now onwards all the Lie groups are assumed to be connected.
\begin{defn}
A complex Lie group $G$ is called {\it Cousin group} if any holomorphic function on it is constant.
\end{defn}

This concept goes back to the 1910 Cousin's paper \cite{C}. The name is due to Huckleberry and Margulis \cite{HM}. The definition of Cousin groups appeared in the paper of Morimoto \cite{M} under the name of {\it H.C groups} and in the paper of Kopfermann \cite{K} under the name of {\it toroidal groups}. For a rather complete account on Cousin groups, see \cite{AK}.

\begin{pr}
\label{HahnBanach}
Any representation of a Cousin group $G$ in a Banach space $V$ is trivial.
\end{pr}
\begin{proof}
A representation of $G$ in a Banach space $V$ is a map $G\to\End(V)$. Pullback of any linear functional on $\End(V)$ to $G$ is constant, so, by Hahn--Banach theorem, image of $G$ is a single point.
\end{proof}

\begin{pr}
\label{Verbitsky}
Finite-dimensional (Banach) Cousin groups are commutative. Moreover, they are quotients of a complex vector (Banach) space by a discrete subgroup. 
\end{pr}
\begin{proof}
By Proposition \ref{HahnBanach}, the adjoint representation of a Cousin group $G$ in its Lie algebra $\fg$ is trivial. That means that $G$ is commutative. Since the Lie algebra $\fg$ is abelian, the exponential map $\fg\to G$ is a homomorphism. Its kernel is discrete since the quotient is a manifold of the same dimension. 
\end{proof}

Of course, complex tori are Cousin groups. Nevertheless, Cousin groups need not to be compact. For example, a quotient of $\bC^2 = \bC\langle v,w\rangle$ by a rank $3$ lattice which spans the real subspace $\{(v,w)\mid\mathrm{Im}~w=0\}$ but intersects the complex subspace $\{w=0\}$ only by $\{0\}$, is a Cousin group.

\begin{pr}[K.~Kopfermann, 1964 \cite{K}]
\label{Kopfermann}
Any commutative complex Lie group is a product of Cousin group, several copies of $\bC$ and several copies of $\GL(1,\bC)$.
\end{pr}

\begin{pr}
\label{P}
Commutative complex Lie group is Cousin if and only if it does not have nontrivial characters.
\end{pr}
\begin{proof}
One direction follows from Proposition \ref{HahnBanach}. The other way around, by Kopfermann's theorem, any commutative complex Lie group $G$ is isomorphic to the product of a Cousin group and a group $\GL(1,\bC)^n\times\bC^m$, which admits a homomorphism to $\GL(1,\bC)^{n+m}$ via exponent provided $n+m>0$ (that is, $G$ is not Cousin).
\end{proof}

\section{Complex tori and Hodge structures}
This chapter is intended to be an exposition of the well-known correspondence between the complex tori and Hodge structures.

\begin{defn}
Let $k\subseteq\R$ be a subring. A {\it $k$-Hodge structure of weight $n$} is a pair $(V_k,\{V^{p,q}\}_{p+q=n}^{p,q \geq 0})$ of a finite generated $k$-module $V_k$ and a splitting $V_k\otimes_k\bC=\oplus_{p+q=n}V^{p,q}$ such that $\overline{V^{p,q}}=V^{q,p}$.
\end{defn}

\begin{pr}
\label{k-k'}
Let $k\subseteq k'\subseteq\R$ be subrings. The category of $k$-Hodge structures is equivalent to category of $k'$-Hodge structures on modules of form $V_k\otimes k'$ with $k$-linear maps, where $V_k$ are $k$-modules.
\end{pr}
\begin{proof}
Follows easily from the above Definition.
\end{proof}

\begin{pr}
\label{Jacobi}
Let $V_\Z$ be a $\Z$-Hodge structure of weight $1$. Denote by $\pi$ the projection of $V_\bC$ onto $V^{1,0}$ along $V^{0,1}$. Then the quotient $V^{1,0}/\pi(V_\Z)$ is a complex torus.
\end{pr}
\begin{proof}
Of course $\pi \colon V_\R \to V^{1,0}$ is an isomorphism of $\R$-vector spaces. One has $\R\langle\pi(V_\Z)\rangle = \pi(\R\langle V_{\Z} \rangle) = \pi(V_\Z\otimes\R) = \pi(V_\R) = V^{1,0}$. So $\pi(V_\Z)$ has full rank inside $V^{1,0}$ considered as an $\R$-vector space, and the quotient is a compact torus. It inherits complex structure from $V^{1,0}$.
\end{proof}

\begin{pr}
\label{r_1_equals_cpx}
The category of complex vector spaces is equivalent to the category of $\R$-Hodge structures of weight $1$.
\end{pr}
\begin{proof}
An $\R$-Hodge structure of weight $1$ is an $\R$-vector space $V_\R$ together with a splitting $V_\bC=V_\R\otimes\bC = V^{1,0}\oplus V^{0,1}$ such that $V^{0,1}=\overline{V^{1,0}}$. Define $I\in\End(V_\bC)$ as $\sqrt{-1}\id_{V^{1,0}}\oplus(-\sqrt{-1})\id_{V^{0,1}}$. This operator is real, that is, preserves $V_\R=V_\bC^{\Gal(\bC:\R)}$, and $I^2=-\id$. The other way around, any $I\in\End(V_\R)$ with $I^2=-\id$ determines an $\R$-Hodge structure on $V_\R$ by setting $V^{1,0}$ to be the $\sqrt{-1}$-eigenspace of $I$.

It can be easily verified that maps of Hodge structure give rise to maps of vector spaces compatible with the complex structures operator, and vice versa.
\end{proof}

\begin{pr}
\label{complex}
Suppose $V_\R$ is an $\R$-Hodge structure of weight $1$. Then the projection $V_\R=V_\bC^{\Gal(\bC:\R)} \to V^{1,0}$ is complex linear w.~r.~t. the complex structure on $V_\R$ corresponding to the Hodge structure.
\end{pr}
\begin{proof}
The complex structure operator on $V_\R=\{v+\overline{v}\mid v\in V^{1,0}\}$ is given by $I(v+\overline{v}) = \sqrt{-1}(v - \overline{v}) = \sqrt{-1}v + \overline{\sqrt{-1}v}$. Of course the projection maps $I(v+\overline{v})$ to $\sqrt{-1}v$.
\end{proof}

\begin{pr}
\label{z_1_equals_tori}
The category of complex tori is equivalent to the category of $\Z$-Hodge structures of weight $1$.
\end{pr}
\begin{proof}
The correspondence on objects in one direction is given by Proposition \ref{Jacobi}. Let us construct it in the backward direction.

Let $X$ be a complex torus. Then its universal cover $\widetilde{X}$ is a complex vector space, and the kernel of the projection $\ker(\widetilde{X} \to X)$ is a lattice which is of full rank if one consider $\widetilde{X}$ as an $\R$-vector space, that is, a $\Z$-module $V_\Z$ with complex structure operator on $V_\Z\otimes\R$. By Propositions \ref{r_1_equals_cpx} and \ref{k-k'} it is the same as $\Z$-Hodge structure of weight $1$. It follows from the Proposition \ref{complex} that these correspondences are mutually inverse.

Let $f \colon X \to Y$ be a map of complex tori. Its lift to the universal cover $\widetilde{f}\colon\widetilde{X}\to\widetilde{Y}$ is a $\bC$-linear map, mapping the lattice $\ker(\widetilde{X} \to X)$ to the lattice $\ker(\widetilde{Y} \to Y)$. By Propositions \ref{r_1_equals_cpx} and \ref{k-k'} this is a map of $\Z$-Hodge structures of weight~$1$.
\end{proof}

Another, in fact equivalent way to obtain a weight $1$ $\Z$-Hodge structure from the complex torus, is to take its first homology with integral coefficients.

\begin{defn}
A surjective morphism of complex tori with finite kernel is called an {\it isogeny}.
\end{defn}

\begin{defn}
For a category $C$ and a class of morphisms $W$ the {\it localization} $C[W^{-1}]$ is the category with the same class of objects and $$\Hom_{C[W^{-1}]}(X,Y) = \{X \xleftarrow{f_1} Z_1 \to Y_1 \xleftarrow{f_2} Z_2 \to \dots \xleftarrow{f_n} Z_n \to Y \mid f_1,f_2\dots,f_n\in W\}$$ up to suitable equivalence with obvious composition.
\end{defn}

Essential details of the definition of localization can be found in any text on categorical algebra, e.~g. \cite{KSch}. However, they are insufficient for our exposition. 

\begin{defn}
Localization of the category of complex tori by isogenies is called the {\it category of complex tori up to isogeny}.
\end{defn}

\begin{pr}
\label{q_1_equals_tori_isog}
The category of complex tori up to isogeny is equivalent to the category of $\Q$-Hodge structures of weight $1$.
\end{pr}
\begin{proof}
Any $\Q$-Hodge structure may be made $\Z$-Hodge structure by taking a base $\{v_i\}$ of $V_\Q$ and setting $V_\Z=\Z\langle\{v_i\}\rangle$. Hence the objects of this category can be identified with complex tori by Proposition \ref{z_1_equals_tori}, although in a non-unique way. However, the isogeny class of a torus corresponding to $V_\Z$ is independent on choise of the base. Indeed, tori corresponging to bases $\{v_i\}$ and $\{Nv_i\}$, where $N$ is a nonzero integer, are isogenous by construction, and any two bases of a $\Q$-vector space can be made to span the same lattice by multiplying them by appropriate integers. Moreover, any morphism of $\Q$-Hodge structures may be made a morphism of $\Z$-Hodge structures in this way.
\end{proof}

\section{Hodge structures from CR geometric viewpoint}

\subsection{Reminder of CR and co-CR vector spaces}
We start this section with a reminder of definitions in CR geometry. For a more detailed exposition, see e.~g. \cite{MOP}.

\begin{defn}
A {\it CR vector space} $(V,H,I)$ is an $\R$-vector space $V$ together with a subspace $H\subseteq V$ (called {\it distinguished complex subspace}) and an operator $I\colon H\to H$ such that $I^2=-1$ (called {\it CR structure operator}). A {\it co-CR vector space} $(U,F,J)$ is an $\R$-vector space with a subspace $F\subseteq U$ (called {\it distinguished cocomplex subspace}) and an operator $J\colon U/F\to U/V$ such that $J^2=-1$ (called {\it co-CR structure operator}). The {\it (co-)CR linear map} of (co-)CR vector spaces is a linear map of underlying $\R$-vector spaces mapping distinguished (co)complex subspace to distinguished (co)complex subspace and whose restriction to distinguished complex subspace (map induced on quotients by distinguished cocomplex subspaces) is complex linear w.~r.~t. the (co-)CR structures operators.
\end{defn}

\begin{pr}
\label{example}
Suppose that $W$ is a $\bC$-vector space, and $V\subseteq W$ is an $\R$-vector subspace. Then $V$ carries a natural CR structure, and $W/V$ carries a natural co-CR structure. \end{pr}
\begin{proof}
The subspace $H=V\cap\sqrt{-1}V$ is invariant under multiplication by $\sqrt{-1}$ and is hence a $\bC$-vector subspace in the space $V$. The pair $(V,H)$ is a CR vector space.

Denote the projection of the space $\sqrt{-1}V$ to $W/V$ by $F$. Then the quotient $(W/V)/F$ is isomorphic to $W/\bC V$, hence carries a natural complex structure. The pair $(W/V,F)$ is a co-CR vector space.
\end{proof}

The complex structure operator is characterized by its $\sqrt{-1}$-eigensubspace, which leads to another 

\begin{defn}
A {\it CR vector space} $(V,H^{1,0})$ is an $\R$-vector space $V$ together with a subspace $H^{1,0}\subseteq V\otimes\bC$ such that $H^{1,0}\cap\overline{H^{1,0}}=\{0\}$. A {\it co-CR vector space} $(U,F^{1,0})$ is an $\R$-vector space $U$ together with a subspace $F^{1,0}\subseteq U\otimes\bC$ such that $\langle F^{1,0},\overline{F^{1,0}}\rangle=U\otimes\bC$.
\end{defn}

In the notation of the previous Definition, $H=(H^{1,0}\oplus\overline{H^{1,0}})^{\Gal(\bC:\R)}$ and $F=F^{1,0}\cup\overline{F^{1,0}}$. These two definitions of CR structures are obviously equivalent; to see that the definitions of co-CR structures are equivalent, note that $F^{1,0}$ may be reconstructed as a preimage of the subspace $(U/F)^{1,0}\subseteq(U/F)\otimes\bC$. 

We may see from this Definition that any CR or co-CR structure is isomorphic to an example from Proposition \ref{example}, however, strictly speaking, in a non-unique way.

\subsection{Bi-CR structures and Hodge structures of weight two}
Of course, if $(V,V^{1,0})$ is both CR and co-CR structure, then $V\otimes\bC=V^{1,0}\oplus\overline{V^{1,0}}$, and this is just a complex structure. But there exists a more appropriate notion of compatibility of a CR and a co-CR structure on the same $\R$-vector space.

\begin{defn}
We call by a {\it bi-CR vector space} $(V,H,F,I,J)$ a pair of CR structure $(V,H,I)$ and co-CR structure $(V,F,J)$ such that $H\oplus F=V$ is a direct sum decomposition and the projection $V/F\to H$ is complex linear w.~r.~t. the complex structure $J$ on the left and $I$ on the right. Alternatively, a {\it bi-CR vector space} $(V,H^{1,0},F^{1,0})$ is a pair of CR structure $(V,H^{1,0})$ and co-CR structure $(V,F^{1,0})$ such that $H^{1,0}\subseteq F^{1,0}$ and $F^{1,0}\cap\overline{H^{1,0}}=0$.
\end{defn}

\begin{pr}
\label{r_2_equals_bi-CR}
The category of bi-CR vector spaces is equivalent to the category of $\R$-Hodge structure of weight $2$.
\end{pr}
\begin{proof}
Let $V$ be an $\R$-Hodge structure of weight 2: $V\otimes\bC=V^{2,0}\oplus V^{1,1}\oplus V^{0,2}$. Then $(V,V^{2,0},V^{2,0}\oplus V^{1,1})$ is a bi-CR structure. The other way around, if $(V,H,F)$ is a bi-CR structure, then $V\otimes\bC = H^{1,0} \oplus (F\otimes\bC) \oplus \overline{H^{1,0}}$ is an $\R$-Hodge structure of weight 2. It can be easily seen that these correspondences are compatible with maps.
\end{proof}

Similarly, to a $\Z$-Hodge structure of weight 2 $V_\Z$ one can associate a torus with left invariant bi-CR structure $A_{V_\Z}=V_\R/V_\Z$. 

\begin{defn}
Let $V_\Z$ be a $\Z$-Hodge structure of weight 2. We call the quotient $A_{V_\Z} = V_\R / V_\Z$ with canonical left invariant bi-CR structure an {\it Abel torus} of $V_\Z$.
\end{defn}

\begin{pr}
\label{z_2_equals_bi-cr_tori}
The category of bi-CR tori is equivalent to the category of $\Z$-Hodge structures of weight 2.
\end{pr}
\begin{proof}
Much like the proof of the Proposition \ref{z_1_equals_tori}, this proof follows from Propositions \ref{r_2_equals_bi-CR} and \ref{k-k'}.
\end{proof}

\subsection{The canonical embedding}

\begin{defn}
Let $V_\R$, $V_\R\otimes\bC=V^{2,0}\oplus V^{1,1}\oplus V^{0,2}$ be an $\R$-Hodge structure of weight 2. Denote by $\varpi$ the projection $V^{2,0}\oplus V^{0,2}\to V^{2,0}$, and by $\imath$ the embedding $(V^{1,1})^{\Gal(\bC:\R)}\to V^{1,1}$. We call the map $$\gamma=\varpi|_{(V^{2,0} \oplus V^{0,2})^{\Gal(\bC:\R)}}\oplus\imath\colon(V^{2,0} \oplus V^{0,2})^{\Gal(\bC:\R)}\oplus(V^{1,1})^{\Gal(\bC:\R)}=V_\R\to V^{2,0} \oplus V^{1,1}$$ the {\it canonical embedding}. 
\end{defn}

\begin{pr}The canonical embedding is CR linear w.~r.~t. the usual complex structure on $V^{2,0}\oplus V^{1,1}$ considered as CR structure.
\end{pr}
\begin{proof}
The proof is more or less straightforward. The distinguished complex subspace of the CR space $V_\R$ is the subspace $(V^{2,0} \oplus V^{0,2})^{\Gal(\bC:\R)}$ and consists of vectors $v+\overline{v}$ for $v\in V^{2,0}$ (moreover, each vector in the distinguished complex subspace of $V_\R$ can be written in such way uniquely). The CR structure operator is given by $I(v+\overline{v}) = \sqrt{-1}(v-\overline{v})$. The projection $\varpi$ maps $v+\overline{v}$ to $v\in V^{2,0}$, and $Iv$ to $\sqrt{-1}v$. Hence the sum $\varpi\oplus\imath$ is CR linear.
\end{proof}

\begin{pr}
\label{iso}
Consider the induced CR structure on the space $\gamma(V_\R)\subseteq V^{2,0} \oplus V^{1,1}$. W.~r.~t. this CR structure on the image, the canonical embedding is a CR isomorphism.
\end{pr}
\begin{proof}
One has $\gamma(V_\R)\cap\sqrt{-1}\gamma(V_\R) = (V^{2,0} \oplus (V^{1,1})^{\Gal(\bC:\R)})\cap(\sqrt{-1}V^{2,0} \oplus \sqrt{-1}(V^{1,1})^{\Gal(\bC:\R)}) = (V^{2,0}\cap\sqrt{-1}V^{2,0}) \oplus (V^{2,0}\cap\sqrt{-1}(V^{1,1})^{\Gal(\bC:\R)}) \oplus ((V^{1,1})^{\Gal(\bC:\R)}\cap\sqrt{-1}V^{2,0}) \oplus ((V^{1,1})^{\Gal(\bC:\R)}\cap\sqrt{-1}(V^{1,1})^{\Gal(\bC:\R)})$. The first summand equals $V^{2,0}$ as it is complex vector subspace. The second and third summand vanish, because $V^{2,0} \cap V^{1,1} = \{0\}$ by definition of Hodge structure. The fourth summand vanishes because for any vector $u$ in this summand one has $u = -\sqrt{-1}\sqrt{-1}u = -\sqrt{-1}\overline{\sqrt{-1}u} = -\sqrt{-1}(-\sqrt{-1})\overline{u} = -\overline{u} = -u$. Therefore the distinguished complex subspace in $\gamma(V_\R)$ in induced CR structure equals $V^{2,0}$. However, the image of distinguished complex subspace in $V_\R$ is precisely $V^{2,0}$ by the previous Proposition.
\end{proof}

\section{Hodge Cousin groups}
The Abel tori, defined in the previous Section, are generalizations of Abel--Jacobi varieties for $\Z$-Hodge structures of weight 1 to $\Z$-Hodge structures of weight 2. In the present section we propose another generalization of Abel--Jacobi varieties, which works for $\Z$-Hodge structures of any weight.

\begin{defn}
Let $V_\Z$ be a $\Z$-Hodge structure of weight $n$. We call {\it Jacobi group} the complex Lie group $$J_{V_\Z} = \frac{\oplus_{p \leq q}V^{p,q}}{\pi(V_\Z)},$$ where $\pi$ denotes the projection of $V_\bC$ onto $\oplus_{p \leq q}V^{p,q}$ along the subspace $\oplus_{p>q}V^{p,q}$.
\end{defn}

\begin{pr}
\label{Euler}
Let $W_\Z$ be a weight 0 $\Z$-Hodge structure. Then $J_{W_\Z}$ is isomorphic to a product of several copies of $\GL(1,\bC)$.
\end{pr}
\begin{proof}
One has $W^{0,0} = W_\bC = W_\Z\otimes\bC$. Pick up an integral basis $w_i$ for the lattice $W_\Z$, it determines a decomposition $W^{0,0}=\oplus_i\bC w_i$. Now one has $J_{W_\Z} = \frac{\oplus_{p \leq q}W^{p,q}}{\pi(W_\Z)} = \frac{W^{0,0}}{W_\Z} = \frac{\oplus_i\bC w_i}{\oplus_i\Z w_i} = \oplus_i(\bC/\Z) = \oplus_i\GL(1,\bC)$.
\end{proof}

\begin{defn}[L.~Ornea, M.~Verbitsky, V.~Vuletescu \cite{OVV}]
The Cousin group which is of the form $J_{V_\Z}$ for some $\Z$-Hodge structure $V_\Z$ is called {\it Hodge Cousin group}.
\end{defn}

\begin{pr}
\label{criterion}
The complex Lie group $J_{V_\Z}$ is Cousin group if and only if the Hodge structure $V_\Z$ does not admit a nontrivial quotient Hodge structure of weight $0$.
\end{pr}
\begin{proof}
Suppose that $J_{V_\Z}$ is a Cousin group. Let $V_\Z \to W_\Z$ be a surjective map onto a $\Z$-Hodge structure of weight $0$. Then the induced homomorphism of complex Lie groups $J_{V_\Z}\to J_{W_\Z} = \oplus_i\GL(1,\bC)$, determines, by Proposition \ref{Euler}, a collection of characters of $J_{V_\Z}$. By Propostion \ref{HahnBanach} all representations of the Cousin group $J_{V_\Z}$ are trivial, hence the $\Z$-Hodge structure $W_\Z$ is trivial. 

Suppose that $J_{V_\Z}$ is not a Cousin group. Since it is commutative, by Proposition \ref{P}, it admits a nontrivial character $J_{V_\Z}\to\bC/\Z$. The lift of this map to the universal cover is a surjective map $V_\bC\to\bC$. Its restriction to $V_\Z \subset V_\bC$ is a surjective map of $V_\Z$ to $\Z \subset \bC$ a $\Z$-Hodge structure of weight $0$.
\end{proof}

\begin{pr}
Let $V_\Z$ be a $\Z$-Hodge structure of weight 2. Then the canonical embedding $\gamma$ descends to an embedding of the Abel torus $A_{V_\Z}$ to the Jacobi group $J_{V_\Z}$.
\end{pr}
\begin{proof}
Let $v \in V_\Z$ be a vector considered as a vector in $V_\R$. It can be uniquely written as a sum $v'+w$, where $v'\in(V^{2,0} \oplus V^{0,2})^{\Gal(\bC:\R)}$ and $w\in(V^{1,1})^{\Gal(\bC:\R)}$. It can be further written as $u+\overline{u}+w$, where $u\in V^{2,0}$. One has $\gamma(v) = u + w$. On the other hand, $u + w$ is precisely the $\left((2,0)+(1,1)\right)$-component of the vector $v$ considered as a vector in $V_\bC$, that is, $\pi(v)$. Hence $\gamma(V_\Z)=\pi(V_\Z)$, and the map $\gamma$ descends to the quotients by $V_\Z$ and $\pi(V_\Z)$, respectively.
\end{proof}

\begin{pr}
\label{maximality}
The image $\gamma(A_{V_\Z})$ is the maximal compact subgroup of $J_{V_\Z}$.
\end{pr}
\begin{proof}
The real dimension of maximal compact subgroup in $J_{V_\Z}$ equals to $\dim V_\R$, just like the dimension of $A_{V_\Z}=\gamma(A_{V_\Z})$. On the other hand, $\gamma(A_{V_\Z})$ is compact.
\end{proof}

\begin{pr}
The maximal compact subgroup of the complex Lie group $J_{V_\Z}$ with the induced CR structure on it is CR isomorphic to the bi-CR group $A_{V_\Z}$.
\end{pr}
\begin{proof}
By Proposition \ref{maximality}, the maximal compact subgroup of $J_{V_\Z}$ is the subgroup $\gamma(A_{V_\Z})$. The induced CR structure on it is isomorphic (via $\gamma$) to the CR structure on $A_{V_\Z}$ by Proposition \ref{iso}.
\end{proof}

Suppose that $G$ is a Cousin group. If it is Hodge Cousin group of some $\Z$-Hodge structure of weight 2 $V_\Z$, then the summands $V^{2,0}$ and $V^{0,2}$ can be reconstructed from its complex structure alone: namely, the maximal compact subgroup $G_c \subseteq G$ carries a CR structure, which rises to a (linear) CR structure on its universal cover $\widetilde{G_c}$. If $H\subseteq\widetilde{G_c}$ is its distinguished complex subspace, then there exists an isomorphism $\widetilde{G_c}\otimes\bC \to V_\bC$ which identifies $H^{1,0}$ with $V^{2,0}$. This observation leads to the following
\begin{probl}
Describe the Hodge Cousin groups among the Cousin groups in terms of their complex Lie group structure.
\end{probl}

\section{Polarization}
The main results of the present section are contained in the unpublished paper \cite{OVV} by L.~Ornea, M.~Verbitsky and V.~Vuletescu and have been presented in a talk \cite{V} given by Verbitsky in the University of Nottingham at September~14, 2016. However, they motivated the above work, so we prove them here.

\subsection{Reminder on abelian varieties}
\begin{defn}
A {\it polarization} on the $k$-Hodge structure of weight $n$ $V$ is a $k$-valued bilinear form $Q$ on $V_k$ such that 
\begin{enumerate}
\item $Q(v,u)=(-1)^nQ(u,v)$, 
\item $Q(u,v)=0$ for $u\in V^{p,q}$, $v\in V^{r,s}$ with $p\neq s$, $q\neq r$,
\item $\sqrt{-1}^{p-q}Q(u,\overline{u})>0$ for nonzero $u\in V^{p,q}$.
\end{enumerate}
Maps between polarized Hodge structures are orthogonal maps of underlying Hodge structures.
\end{defn}

Let me remind some well-known facts.

\begin{pr}
\label{polar_r_1_equals_herm}
The category of polarized $\R$-Hodge structures of weight $1$ is equivalent to the category of Hermitian vector spaces.
\end{pr}
\begin{proof}
For an $\R$-Hodge structure of weight 1 $V$ (which is by Proposition \ref{r_1_equals_cpx} the same as a complex structure operator on $V_\R$) a polarization $Q$ is skew-symmetric. If $u,v \in V^{1,0}$, then one has $Q(I(u+\overline{u}),I(v+\overline{v})) = Q(\sqrt{-1}(u-\overline{u}),\sqrt{-1}(v-\overline{v})) = -Q(u-\overline{u},v-\overline{v}) = Q(u,\overline{v}) + Q(\overline{u},v) = Q(u+\overline{u},v+\overline{v})$, that is, $I$ is orthogonal w.~r.~t. $Q$. If we define $g(x,y)=Q(Ix,y)$ for $x,y\in V_\R$, then $g(y,x) = Q(Iy,x) = -Q(x,Iy) = -Q(Ix,I^2y) = -Q(Ix,-y) = Q(Ix,y) = g(x,y)$, go $g$ is symmetric. Moreover, for $0 \neq x = u+\overline{u}$, $u\in V^{1,0}$ one has $g(x,x) = Q(I(u+\overline{u}),u+\overline{u}) = Q(\sqrt{-1}(u-\overline{u}),u+\overline{u}) = \sqrt{-1}(Q(u,\overline{u}) - Q(\overline{u},u)) = 2\sqrt{-1}Q(u,\overline{u}) > 0$, so the form $g$ is positive definite, and the form $h=g+\sqrt{-1}Q$ is Hermitian.

The other way around, imaginary part of any Hermitian form gives rise to a polarization on the correspondent weight $1$ Hodge structure.
\end{proof}

\begin{defn}
A complex torus with a left invariant K\"ahler form $\omega$ is called an {\it abelian variety}, if $[\omega]$ is an integral cohomology class.
\end{defn}

By Kodaira's embedding theorem, it is the same that a projective complex torus.

\begin{pr}
\label{polar_z_1_equals_abel}
The category of polarized $\Z$-Hodge structures of weight $1$ is equivalent to the category of abelian varieties.
\end{pr}
\begin{proof}
By Propositions \ref{z_1_equals_tori} and \ref{polar_r_1_equals_herm}, the category of polarized $\Z$-Hodge structures of weight $1$ is equivalent to the category of complex tori with left-invariant Hermitian metric satisfying certain integrality conditions. 

Consider a Hermitian torus $(X,\omega)$ which is correspondent to a polarized $\Z$-Hodge structure of weight $2$ $V$. The $2$-form $\omega$ is parallel w.~r.~t. Levi-Civita connection and is hence closed, so $(X,\omega)$ a K\"ahler torus. We need to show that the K\"ahler class $[\omega]$ is integral. For any two vectors $u,v \in V_\Z$ consider an oriented parallelogram spanned by them in $V_\R$, and denote by $c_{u,v}$ the $2$-cycle in $X$, which is the projection of this parallelogram into $X=V_\R / V_\Z$. Homology classes of such cycles $[c_{u,v}]$ generate the group $H_2(X,\Z)$. On the other hand, it is easy to see that $\int_{c_{u,v}}\omega = Q(u,v)$. As the form $Q$ is integral, the class $[\omega]$ is also integral, and $(X,\omega)$ is an abelian variety.
\end{proof}

\begin{pr}
\label{polar_q_1_equals_abel_isog}
The category of polarized $\Q$-Hodge structures of weight $1$ is equivalent to the category of abelian varieties up to isogeny.
\end{pr}
\begin{proof}
This is a combination of Proposition \ref{q_1_equals_tori_isog} and \ref{polar_z_1_equals_abel}.
\end{proof}

\begin{pr}[H.~Poincar\'e, 1886 \cite{P}]The category of abelian varieties up to isogeny is semisimple.
\end{pr}
\begin{proof}
The category of $\Q$-Hodge structures of weight $1$ is semisimple because one may take orthogonals w.~r.~t. $Q$. However, by Proposition \ref{polar_q_1_equals_abel_isog} this category is equivalent to the category of abelian varieties up to isogeny.
\end{proof}

This Poincar\'e complete reducibility theorem, stated in less fancy terms, says that for any abelian subvariety $A'\subseteq A$ inside the abelian variety $A$ there exists an abelian variety $A''\subseteq A$ such that $A$ is isogenous to $A' \times A''$.

\subsection{Abelian Cousin groups}
Now we want to generalize these facts to the Hodge structures of weight $2$.

\begin{defn}
We call a pair $((V, H, F, I, J),g)$ consisting of a bi-CR vector space $(V, H, F, I, J)$ and a pseudo-Euclidean metric $g$ on $V$ {\it CR Hermitian}, if 
\begin{enumerate}
\item $g|_H$ is negative definite and $I$-invariant,
\item $g|_F$ is positive definite;
\item $H\perp^gF$.
\end{enumerate}
\end{defn}

\begin{pr}
\label{polar_r_2_equals_cr_herm}
The category of polarized $\R$-Hodge structure of weight $2$ is equivalent to the category of CR Hermitian vector spaces.
\end{pr}
\begin{proof}
By Proposition \ref{r_2_equals_bi-CR}, the bi-CR vector spaces and $\R$-Hodge structures of weight $2$ are the same. Let us examine what new structure comes up on bi-CR vector space $(V,H,F)$ along with polarization $Q$ on the corresponding Hodge structure $V_\bC=V^{2,0} \oplus V^{1,1} \oplus V^{0,2}$. This is certainly some bilinear form on the underlying $\R$-vector space $V=V_\R$, so let us denote by the same letter $Q$.

First, $Q$ is symmetric since $Q(x,y) = (-1)^2Q(y,x)$.

Second, for $x,y\in H$, where $x=u+\overline{u}$, $y=v+\overline{v}$ for $u,v\in V^{2,0}$, one has $Q(Ix,Iy) = Q(\sqrt{-1}(u-\overline{u}),\sqrt{-1}(v-\overline{v})) = -Q(u-\overline{u},v-\overline{v}) = -(-Q(u,\overline{v}) - Q(\overline{u},v)) = Q(u,\overline{v}) + Q(\overline{u},v) = Q(u+\overline{u},v+\overline{v}) = Q(x,y)$. If $x\neq 0$, then one has $Q(x,x) = Q(u+\overline{u},u+\overline{u}) = Q(u,\overline{u})+Q(\overline{u},u) = 2Q(u,\overline{u}) < 0$. If $0\neq x\in F=(V^{1,1})^{\Gal(\bC:\R)}$, then $Q(x,x)=\sqrt{-1}^{1-1}Q(x,x)>0$.

Finally, $H$ and $F$ are perpendicular because $V^{2,0}\oplus V^{0,2}$ and $V^{1,1}$ are.
\end{proof}

As we want to speak about pseudo-Riemannian manifolds, we need to introduce an invariant which would replace the usual notion of length of the curve.

\begin{defn}
Let $(X,g)$ be a pseudo-Riemannian manifold and $\gamma\colon[0;1]\to X$ be a curve. Then we call the value $\int_0^1 g(d\gamma(\partial_t),d\gamma(\partial_t))dt$ as {\it action} of $\gamma$.
\end{defn}

This terminology is borrowed from classical mechanics, see e.~g. \cite{B}.

\begin{defn}
Let $(X,g)$ be a pseudo-Riemannian manifold. We call it {\it polarized}, if the action of any closed geodesic on it is a rational number.
\end{defn}

For bi-CR tori an equivalent definition of polarization was given in the paper \cite{CT} by P.~Caressa and A.~Tomassini.

\begin{pr}
\label{polar_q_2_equals_cr_herm_tori}
The category of polarized $\Q$-Hodge structures of weight $2$ is equivalent to the category of left-invariant polarized CR Hermitian tori up to isogeny.
\end{pr}
\begin{proof}
By Propositions \ref{polar_r_2_equals_cr_herm}, \ref{z_2_equals_bi-cr_tori} and \ref{q_1_equals_tori_isog} the category of polarized $\Q$-Hodge structures of weight $2$ is equivalent to the category of left-invariant CR Hermitian tori up to isogeny with some integrality conditions satisfied. Let us figure out what these conditions are. 

Given a polarized $\Q$-Hodge structure of weight $2$ $V_\Q$, fix some polarized $\Z$-Hodge structure of weight $2$ $V_\Z$ such that $V_\Q=V_\Z\otimes\Q$. Denote the corresponding pseudo-Riemannian metric on the torus $A_{V_\Z}$ by $g$. The value of the symmetric form $Q$ is determined by the pseudo-norm of vectors on the lattice $V_\Z$ w.~r.~t. $Q$, and $Q$ is rational if and only if all the values $Q(v,v)$ are. The projection of a segment $[0,v]\subset V_\R$ to $A_{V_\Z} = V_\R / V_\Z$ is a closed geodesic, and square of its length w.~r.~t. the pseudo-Riemannian metric $g$ equals $Q(v,v)$. Hence the desired integrality condition is the rationality of squares of lengths of all closed geodesics.
\end{proof}

This Proposition may be used for a nice derivation of description of the moduli space of polarized CR Hermitian tori given in \cite{CT}: indeed, it is the same as the well-known moduli space of polarized Hodge structures.

CR vector spaces arise as real subspaces in complex, and the following Definition asserts that CR Hermitian vector spaces arise as real subspaces in Hermitian.

\begin{defn}
For an $\R$-vector subspace $V \subseteq U$ of a Hermitian vector space $(U,h)$ we call the CR Hermitian structure $g$ defined by $g|_{V\cap\sqrt{-1}V}=-\mathrm{Re}~h|_{V\cap\sqrt{-1}V}$ and $g|_{(V\cap\sqrt{-1}V)^\perp} = \mathrm{Re}~h|_{(V\cap\sqrt{-1}V)^\perp}$ as the {\it induced CR Hermitian structure}. If $Y \subseteq X$ is a $C^\infty$-submanifold in an almost Hermitian manifold $(X,h)$, then we also call the almost CR Hermitian structure on $Y$ given at each point $y \in Y$ as the induced CR Hermitian structure on $T_yY \subseteq T_yX$ as the {\it induced almost CR Hermitian structure}.
\end{defn}

\begin{pr}
\label{cr_Hermitian_extension}
Let $V \subseteq U$ be an $\R$-vector subspace of a $\bC$-vector space $U$, and $H$ be the induced CR structure. Then for any CR Hermitian structure $(V,H,F,I,J,g)$ extending the CR structure $(V,H)$ there exists a Hermitian structure on $U$ which induces this CR Hermitian structure on $V$.
\end{pr}
\begin{proof}
By definition of a bi-CR structure, $\sqrt{-1}F \subset U$ is an $\R$-complement to $V$ inside $U$. Then we can define a symmetric $\R$-bilinear form $\widetilde{g}$ on $U$ by $\widetilde{g}|_H = -g|_H$, $\widetilde{g}|_F = g|_F$ and for $v\in\sqrt{-1}F$ set $\widetilde{g}(v) = g(-\sqrt{-1}v)$. This is a real part of a Hermitian structure which obviously induces on $V$ the CR Hermitian structure $g$.
\end{proof}

\begin{pr}
For a Hodge Cousin group $J_{V_\Z}$ of a polarized $\Z$-Hodge structure of weight $2$ $V_\Z$ there exist a Hermitian metric on $J_{V_\Z}$ such that the induced CR Hermitian metric on the maximal compact subgroup $A_{V_\Z} \subseteq J_{V_\Z}$ coinicides with the structure described by Proposition \ref{polar_q_2_equals_cr_herm_tori}.
\end{pr}
\begin{proof}
The wanted Hermitian metric on the tangent space $T_eJ_{V_\Z}$ is given by Proposition \ref{cr_Hermitian_extension}. Then it can be extended to all the group $J_{V_\Z}$ via the right shifts.
\end{proof}

Of course in this case actions of closed geodesics w.~r.~t. induced CR Hermitian metric on the subgroup $A_{V_\Z} \subseteq J_{V_\Z}$ are rational. This motivates us to give the following 

\begin{defn}
We call an {\it abelian Cousin group} a Cousin group $G$ with left invariant Hermitian metric such that the maximal compact subgroup $G_c$ with its induced CR Hermitian structure is a polarized pseudo-Riemannian manifold.
\end{defn}

Thus we have a particular answer to the Question from the Section 4 of the present paper.
\begin{pr}
\label{answer_to_problem}
Any abelian Cousin group is Hodge Cousin group of a certain $\Z$-Hodge structure of weight $2$.
\end{pr}
\begin{proof}
First reconstruct the $\R$-Hodge structure from the Cousin group $G$. As it was already stated in the Section 4, one needs to take $V_\R = T_e G_c$, where $G_c\subseteq G$ is the maximal compact subgroup, $V^{2,0} = H^{1,0}$, where $H\subseteq T_e G_c$ is the distinguished complex subspace for a CR strcuture on $G_c\subseteq G$, and $V^{0,2}=\overline{V^{2,0}}$. The subspace $V^{1,1}$ can be reconstructed as the complexification of the orthogonal to $H\subseteq T_e G_c$ w.~r.~t. the pseudo-Riemannian metric on $G_c$.

Now note that $V_\R$ is the universal cover of $G_c$, and its kernel is the lattice. Call it $V_\Z$. Then $(V_\Z,Q,V_\bC=V^{2,0}\oplus V^{1,1}\oplus V^{0,2})$ is a polarized $\Z$-Hodge structure of weight $2$. Denote the restriction of the pseudo-Euclidean form on $T_e G_c$ to $V_\Z$ by $Q$. The integrality condition on the Hermitian form on $G$ implies that $Q$ is an integral form, and defines a polarization on the Hodge structure $V_\Z$.

One can easily check that $A_{V_\Z}=G_c$ and $J_{V_\Z}=G$.
\end{proof}

Note that by Proposition \ref{criterion} the constructed Hodge structure $V_\Z$ does not admit a quotient of weight $0$.

\begin{pr}
\label{polar_cousin}
The category of polarized $\Q$-Hodge structures of weight $2$ without quotients of weight $0$ is equivalent to the category of abelian Cousin groups up to isogeny.
\end{pr}
\begin{proof}
This is a combination of Propositions \ref{answer_to_problem} and \ref{criterion}.
\end{proof}

\begin{pr}[Poincar\'e complete reducibility theorem for Cousin groups]
\label{final}
The category of abelian Cousin groups up to isogeny is semisimple.
\end{pr}
\begin{proof}
By Proposition \ref{polar_cousin}, it is equivalent to a full subcategory of $\Q$-Hodge structures category consisting of polarized $\Q$-Hodge structures of weight $2$ which do not admit a quotient of weight $0$. The category of polarized $\Q$-Hodge structures is itself semisimple, so one needs to check that sum of polarized $\Q$-Hodge structures of weight $2$ without weight $0$ quotients do not admit weight $0$ quotients itself. Indeed, if $W$ is a weight $0$ Hodge structure and $0\neq\Hom(U \oplus V,W) = \Hom(U,W)\oplus\Hom(V,W)$, then either $\Hom(U,W)$ or $\Hom(V,W)$ needs to be nonzero.
\end{proof}

\paragraph*{Acknowledgements.} I would like to thank Misha Verbitsky for turning my mind to the problem and reading a draft of this paper. Many thanks to Renat Abugaliev, Alexander Petrov, Kostiantyn Tolmachov, Lev Soukhanov and Bogdan Zavyalov for fruitful discussions.

\noindent {\sc Rodion N. D\'eev\\
Independent University of Moscow\\
119002, Bolshoy Vlasyevskiy Pereulok 11\\ 
Moscow} \\
\tt deevrod@mccme.ru, also:\\
{\sc Department of Mathematics\\
Courant Institute, NYU \\
251 Mercer Street \\
New York, NY 10012, USA,} \\
\tt rodion@cims.nyu.edu


\begin{thebibliography}{\textbf{MOP}}
\bibitem[\textbf{AK}]{AK} Yu.~Abe, K.~Kopfermann. {\it Toroidal Groups: Line Bundles, Cohomology and Quasi-Abelian Varieties}, Lecture Notes in Mathematics, 1759. Springer-Verlag, Berlin, 2001. viii+133 pp.
\bibitem[\textbf{B}]{B} J.-L.~Brylinski. {\it Loop Spaces, Characteristic Classes and Geometric Quantization}, Progress in Mathematics, 107. Birkh\"auser Boston, Inc., Boston, MA (1993) xvi+300 pp.
\bibitem[\textbf{C}]{C} P.~Cousin. {\it Sur les fonctions triplement p\'eriodiques des deux variables}, Acta Math. 33 (1910), pp. 105---232
\bibitem[\textbf{CT}]{CT} P.~Caressa, A.~Tomassini. {\it Complex foliated tori and their moduli spaces}, Commun. Contemp. Math. 7 (2005), no. 3, 311---324. 
\bibitem[\textbf{HM}]{HM} A.~T.~Huckleberry, G.~A.~Margulis. {\it Invariant analytic hypersurfaces}, Invent. Math. 71 (1983) 235---240
\bibitem[\textbf{K}]{K} K.~Kopfermann. {\it Maximale Untergruppen Abelscher komplexer Liescher Gruppen}, Schr. Math. Inst. Univ. M\"unster No. 29 (1964) iii+72 pp.
\bibitem[\textbf{KSch}]{KSch} M.~Kashiwara, P.~Schapira. {\it Sheaves on manifolds}, Grundlehren der Mathematischen Wissenschaften [Fundamental Principles of Mathematical Sciences], 292. Springer-Verlag, Berlin, 1990. x+512 pp.
\bibitem[\textbf{M}]{M} A.~Morimoto. {\it Non-compact complex Lie groups without non-constant holomorphic functions}, Proc. Conf. Complex Analysis at Univ. Minn., Springer Verlag, Berlin, 1965, pp. 257---272
\bibitem[\textbf{MOP}]{MOP} S.~Marchiafava, L.~Ornea, R.~Pantilie. {\it Twistor theory for CR quaternionic manifolds and related structures}, Monatsh. Math. 167 (2012), no. 3-4, pp. 531--545
\bibitem[\textbf{OVV}]{OVV} L.~Ornea, M.~Verbitsky, V.~Vuletescu. {\it Cousin groups and generalized Oeljeklaus--Toma manifolds}, to appear
\bibitem[\textbf{P}]{P} H.~Poincar\'e. {\it Sur les fonctions ab\'eliennes}, Amer. J. Math, 8 (1886) 289---342
\bibitem[\textbf{V}]{V} M. Verbitsky, 
{\em Cousin groups and generalized Oeljeklaus--Toma manifolds},
a talk at a conference ``The different faces of geometry,
a workshop in honour of Fedor Bogomolov,''
University of Nottingham, September 14, 2016,
\url{http://verbit.ru/MATH/TALKS/genOT-Nottingham-2016.pdf}
\end{thebibliography}
\end{document}